\documentclass[11pt]{article}
\usepackage{enumerate}
\usepackage{amssymb,a4wide,latexsym,makeidx,epsfig,fleqn}
\usepackage{amsthm}
\usepackage{amsmath}
\allowdisplaybreaks[4]
\usepackage{enumerate}
\usepackage{graphicx}
\usepackage{color}
\usepackage{epstopdf}
\newtheorem{theorem}{Theorem}[section]

\newtheorem{lemma}[theorem]{Lemma}

\newtheorem{problem}[theorem]{Problem}

\newtheorem{conjecture}[theorem]{Conjecture}

\begin{document}
\textwidth 150mm \textheight 230mm
\setlength{\topmargin}{-15mm}
\title{Sufficient conditions for spanning $k$-trees in tough graphs
\footnote{This work is supported by the National Natural Science Foundations of China (Nos. 12371348, 12201258), the Postgraduate Research \& Practice Innovation Program of Jiangsu Normal University (No. 2025XKT0632), High-Quality Science and Technology Cultivation Project of Jiangsu Normal University (No. JSNUGZL2026069).}}
\author{{ Caili Jia, Yong Lu\footnote{Corresponding author.}}\\
{\small  School of Mathematics and Statistics, Jiangsu Normal University,}\\ {\small  Xuzhou, Jiangsu 221116,
People's Republic
of China.}\\
{\small E-mails: jiacaili0309@163.com, luyong@jsnu.edu.cn}}

\date{}
\maketitle
\begin{center}
\begin{minipage}{120mm}
\vskip 0.3cm
\begin{center}
{\small {\bf Abstract}}
\end{center}
{\small
The toughness of a graph $G$, denoted by $\tau(G)$, is defined by $\tau(G)=$min $\{\frac{|S|}{c(G-S)}:S\subseteq V(G)$ and $c(G-S)\geq2\}$. A graph $G$ is said to be $\tau$-tough if $\tau(G)\geq \tau$.
Let $k\geq2$ be an integer. A tree $T$ is called a $k$-tree if $d_{T}(v)\leq k$ for each $v\in V(T)$, that is, the maximum degree of a $k$-tree is at most $k$. A $k$-tree $T$ is a spanning $k$-tree if $T$ is a spanning subgraph of a connected graph $G$.
In 1989, Win [Graphs Combin. 5 (1989) 201--205] proved that if $\tau(G)\geq\frac{1}{k-2}$, where $k\geq3$, then $G$ contains a spanning $k$-tree. Liu, Fan and Shu [Discrete Math. 348 (2025) 114593] provided a tight sufficient condition based on the spectral condition for connected $\frac{1}{k}$-tough and $\frac{1}{k-1}$-tough graphs to contain a spanning $k$-tree, where $k\geq3$ is an integer.
A natural and interesting problem arises: Can the value of $\tau$ be refined?

When $\frac{1}{k-2}>\tau\geq\frac{1}{k-1}$, we initially establish a lower bound on the size to ensure that a connected $\frac{t}{t(k-2)+1}$-tough graph $G$ contains a spanning $k$-tree, where $k\geq3$ and $t\geq1$ are integers. Meanwhile, we provide two sufficient conditions in terms of spectral radius and signless Laplacian spectral radius for a connected $\frac{t}{t(k-2)+1}$-tough graph $G$ to contain a spanning $k$-tree, where $k\geq3$ and $t\geq1$ are integers. When $t=1$, we obtain the result $\eta=1$ from Liu, Fan and Shu.

 }
\end{minipage}
\end{center}

\section{Introduction }
\hspace{1.3em}
Throughout this paper, we consider only finite, undirected and simple graphs. Let $G=(V(G),E(G))$ be a graph, where $V(G)$ is the vertex set and $E(G)$ is the edge set.
The \emph{order} and \emph{size} of $G$ are denoted by $|V(G)|=n$ and $|E(G)|=e(G)$, respectively. We denote by $d_{G}(v)$ the \emph{degree} of $v\in V(G)$. Denote by $c(G)$ the number of components of a graph $G$.
Let $G_1$ and $G_2$ be two disjoint graphs. The \emph{union} $G_1\cup G_2$ is the graph with vertex set $V(G_1)\cup V(G_2)$ and edge set $E(G_1)\cup E(G_2)$. The \emph{join} $G_1\vee G_2$ is derived from
$G_1\cup G_2$ by joining every vertex of $G_1$ with every vertex of $G_2$ by an edge.

Let $G$ be a graph of order $n$, and let the \emph{adjacency matrix} of $G$ be defined as $A(G)=(a_{ij})_{n\times n}$, where $a_{ij}=1$ if $v_{i}v_{j}\in E(G)$, and $a_{ij}=0$ otherwise.
The \emph{degree diagonal matrix} is the diagonal matrix of vertex degrees of $G$, denoted by $D(G)$. The \emph{signless Laplacian matrix} $Q(G)$ of $G$ is defined by $Q(G)=D(G)+A(G)$.
The eigenvalues of $A(G)$ and $Q(G)$ are called the \emph{eigenvalues} and the \emph{signless Laplacian eigenvalues} of $G$, and denoted by $\rho_{1}(G)\geq \rho_{2}(G)\geq\cdots\geq\rho_{n}(G)$ and $q_{1}(G)\geq q_{2}(G) \geq\cdots\geq q_{n}(G)$, respectively. The largest eigenvalues of $A(G)$ and $Q(G)$ are also called the \emph{spectral radius} and the \emph{signless Laplacian spectral radius} of $G$, and denoted by $\rho(G)$ and $q(G)$, respectively.

For two integers $a$ and $b$ such that $0\leq a\leq b$, a spanning subgraph $F$ of $G$ is called an \emph{$[a,b]$-factor} if $a\leq d_F(v)\leq b$ for each $v\in V(G)$. If $a=1$, then an $[a,b]$-factor is a $[1,b]$-factor. If $b=1$, then a $[1,b]$-factor is a $1$-factor or a perfect matching. Let $k\geq2$ be an integer. A tree $T$ is called a \emph{$k$-tree} if $d_T(v)\leq k$ for each $v\in V(T)$, that is, the maximum degree of a $k$-tree is at most $k$. A $k$-tree $T$ is a \emph{spanning $k$-tree} if $T$ is a spanning subgraph of a connected graph $G$. It is obvious that $G$ admits a spanning $k$-tree if and only if $G$ admits a connected $[1,k]$-factor.

The study of spanning $k$-trees and $[a,b]$-factors has received extensive attention from scholars. Some sufficient conditions for graphs with $[1,2]$-factors were obtained by many researchers \cite{EF, J, JPW, KKK, LP, WYY}. Kim et al. \cite{KOPR}, Kano and Saito \cite{KS}, Zhou, Xu and Sun \cite{ZXS} showed some results for the existence of $[1,b]$-factors in graphs. Zhou and Liu \cite{ZL} studied the relationship between the spectral radius of a connected graph and its odd $[1,b]$-factors, and claimed a lower bound on the existence of odd $[1,b]$-factors via the spectral radius. Many scholars investigated the properties of $[a,b]$-factors in graphs, and provided some graphic parameter conditions for graphs having $[a,b]$-factors \cite{L, LW, M, Z}.

A graph $G$ is \emph{$\tau$-tough} if $|S|\geq\tau\cdot c(G-S)$ for every subset $S\subseteq V(G)$ with $c(G-S)\geq2$, where $c(G)$ is the number of components of a graph $G$. The \emph{toughness} $\tau(G)$ of $G$ is the maximum $\tau$ for which $G$ is $\tau$-tough (taking $\tau(K_n)=\infty$). Hence if $G$ is not a complete graph, $\tau(G)=$min$\{\frac{|S|}{c(G-S)}\}$ where the minimum is taken over all cut sets of vertices in $G$. This concept was first introduced by Chv\'{a}tal \cite{C} as a way of measuring how tightly various pieces of a graph hold together. Chv\'{a}tal proposed the following conjecture.

\noindent\begin{conjecture}\label{conj:1.1.}\cite{C}
Here exists a finite constant $\tau_0$ such that every $\tau_{0}$-tough graph is Hamiltonian.\\
\end{conjecture}

Fan, Lin and Lu \cite{FLL} initially presented a tight sufficient condition in terms of the spectral radius for a connected $1$-tough graph to contain a connected Hamilton cycle.\\

One naturally considers the following problem:\\

\noindent\begin{problem}\label{pro:1.1.}\cite{LFS}
What is a tight spectral condition to guarantee the existence of factors among tough graphs?\\
\end{problem}

Win \cite{W} proved that if $\tau(G)\geq\frac{1}{k-2}$ with $k\geq3$, then $G$ contains a spanning $k$-tree. \\

\noindent\begin{lemma}\label{le:2.1.}\cite{W}\cite{EZ}
Let $k\geq3$ be an integer. If a connected graph $G$ satisfies
\begin{align*}
c(G-S)\leq(k-2)|S|+2,
\end{align*}
for any vertex subset $S$ of $G$, then $G$ admits a spanning $k$-tree.\\
\end{lemma}

In $1989$, Win \cite{W} (see also \cite{EZ} for a short proof) proved a sufficient condition for the existence of a spanning $k$-tree in a connected graph. Lots of scholars presented some sufficient spectral conditions for graphs to possess spanning $k$-tree. Since a spanning $2$-tree is just a Hamilton path, Ning and Ge \cite{NG} provide such condition for $k=2$ in terms of the adjacency spectral radius. In \cite{LSX}, Liu, Shiu and Xue gave a signless Laplacian spectral radius for the existence of a Hamilton path in a connected graph.
For any $k\geq3$, Fan et al. \cite{FGHL} presented a lower bound on the spectral radius of a connected graph $G$ to ensure that $G$ contains a spanning $k$-tree. Zhou and Wu \cite{ZW} showed a distance spectral radius condition which guarantees the existence of a spanning $k$-tree in a connected graph. Zhou, Zhang and Liu \cite{ZZL} studied the connection between the distance signless Laplacian spectral radius and the spanning
$k$-tree in a connected graph and verified an upper bound on the distance signless Laplacian spectral radius in a connected graph $G$ to ensure the existence of a spanning $k$-tree.

Recently, Liu, Fan and Shu \cite{LFS} proposed the following spectral conditions to ensure that connected $\frac{1}{k}$-tough and $\frac{1}{k-1}$-tough graphs contain a spanning $k$-tree, where $k\geq3$ is an integer.\\

\textbf{[Theorem 1.3. in \cite{LFS}]}
Let $G$ be a connected $\frac{1}{k-\eta}$-tough graph of order $n$ with $k\geq3$ and $\eta=\{0, 1\}$. Each of the following holds. \\
\begin{enumerate}[(1)]
\item
If $n\geq 8k+12$ and $\rho(G)\geq\rho(K_{\eta+2}\vee(K_{n-(k-1)(\eta+2)-2}\cup((k-2)(\eta+2)+2)K_{1}))$, then $G$ contains a spanning $k$-tree unless $G\cong K_{\eta+2}\vee(K_{n-(k-1)(\eta+2)-2}\cup ((k-2)(\eta+2)+2)K_{1})$.

\item
If $n\geq11k+47$ and $q(G)\geq q(K_{\eta+2}\vee(K_{n-(k-1)(\eta+2)-2}\cup ((k-2)(\eta+2)+2)K_{1}))$, then $G$ contains a spanning $k$-tree unless $G\cong K_{\eta+2}\vee(K_{n-(k-1)(\eta+2)-2}\cup ((k-2)(\eta+2)+2)K_{1})$.
\end{enumerate}

Based on the above, a natural and interesting problem arises:\\
\noindent\begin{problem}\label{pro:1.2.}
Can the value of $\tau$ be refined?\\
\end{problem}

Concerning Problem \ref{pro:1.2.}, we prove that if $\frac{1}{k-2}>\tau\geq\frac{1}{k-1}$ with $k\geq3$, then $G$ contains a spanning $k$-tree.
We initially establish a lower bound on the size to ensure a connected $\frac{t}{t(k-2)+1}$-tough graph $G$ contains a spanning $k$-tree, where $k\geq3$ and $t\geq1$ are integers.

\noindent\begin{theorem}\label{th:1.1.}
Let $G$ be a connected $\frac{t}{t(k-2)+1}$-tough graph with $n$ vertices, where $k\geq 3$ and $t\geq 1$ are integers. If $n\geq\frac{3tk^3-(9t-2)k^2+(6t-5)k+6}{(k-2)^2}$ and
$$e(G)>\left(
\begin{array}{ccccccccc}
n-3t(k-2)-2 \\
2
\end{array}
\right)+3t(3t(k-2)+2),$$
then $G$ contains a spanning $k$-tree.
\end{theorem}

Meanwhile, we provide two sufficient conditions in terms of spectral radius and signless Laplacian spectral radius for a connected $\frac{t}{t(k-2)+1}$-tough graph $G$ contains a spanning $k$-tree, where $k\geq3$ and $t\geq1$ are integers. When $t=1$, we obtain the result $\eta=1$ from Liu, Fan and Shu.

\noindent\begin{theorem}\label{th:1.2.}
Let $G$ be a connected $\frac{t}{t(k-2)+1}$-tough graph with $n$ vertices, where $k\geq 3$ and $t\geq 1$ are integers. If $n\geq 3kt+k-3t+6$ and
$$\rho(G)\geq\rho(K_{3t}\vee(K_{n-3t(k-1)-2}\cup(3t(k-2)+2)K_1)),$$
then $G$ contains a spanning $k$-tree unless $G\cong K_{3t}\vee(K_{n-3t(k-1)-2}\cup(3t(k-2)+2)K_1)$.
\end{theorem}

\noindent\begin{theorem}\label{th:1.3.}
Let $G$ be a connected $\frac{t}{t(k-2)+1}$-tough graph with $n$ vertices, where $k\geq 3$ and $t\geq 1$ are integers. If $n\geq 3kt^2+k+18t+4$ and
$$q(G)\geq q(K_{3t}\vee(K_{n-3t(k-1)-2}\cup(3t(k-2)+2)K_1)),$$
then $G$ contains a spanning $k$-tree unless $G\cong K_{3t}\vee(K_{n-3t(k-1)-2}\cup(3t(k-2)+2)K_1)$.
\end{theorem}

\section{Proof of Theorem \ref{th:1.1.} }
\hspace{1.3em}

In this section, we give the proof of Theorem \ref{th:1.1.}. Before doing this, we need the following lemma.

\noindent\begin{lemma}\label{le:3.1.}\cite{CLX}
Let $n=\sum\limits_{i=1}^{t}n_{i}+s$. If $n_{1}\geq n_{2}\geq\cdots\geq n_{t}\geq1$ and $n_{1}\leq n-s-t+1$. Then
\begin{align*}
e(K_{s}\vee(K_{n_{1}}\cup K_{n_{2}}\cup\cdots\cup K_{n_{t}}))\leq e(K_{s}\vee(K_{n-s-t+1}\cup(t-1) K_{1})).
\end{align*}
\end{lemma}

\noindent\textbf{Proof of Theorem \ref{th:1.1.}.}

Suppose that a connected $\frac{t}{t(k-2)+1}$-tough graph $G$ has no spanning $k$-tree, where $t\geq1$ is an integer and $k\geq1$. By Lemma \ref{le:2.1.}, we have
\begin{align}
c(G-S)\geq(k-2)|S|+3
\end{align}
for some nonempty subset $S$ of $V(G)$. According to the definition of $\frac{t}{t(k-2)+1}$-tough graphs, we obtain
\begin{align}
\frac{|S|}{c(G-S)}\geq \tau(G)\geq\frac{t}{t(k-2)+1}.
\end{align}
Using $(1)$ and $(2)$, we conclude
$$(t(k-2)+1)|S|\geq t\cdot c(G-S)\geq t((k-2)|S|+3),$$
and hence $|S|\geq3t$. Let $|S|=s$. We know $G$ is a spanning subgraph of $G_1=K_{s}\vee(K_{n_{1}}\cup K_{n_{2}}\cup\cdots\cup K_{n_{(k-2)s+3}})$ for some positive integers $n_{1}\geq n_{2}\geq\cdots\geq n_{(k-2)s+3}$ and $\sum\limits_{i=1}^{(k-2)s+3}n_{i}=n-s$. Therefore, we have
$$e(G)\leq e(G_1).$$
Let $G_2=K_{s}\vee(K_{n-(k-1)s-2}\cup((k-2)s+2) K_{1})$, we obtain
$$e(G_1)\leq e(G_2)$$
by Lemma \ref{le:3.1.}.
Let $G_3=K_{3t}\vee(K_{n-3t(k-1)-2}\cup(3t(k-2)+2) K_{1})$. Note that $e(G_3)=\left(
\begin{array}{ccccccccc}
n-3t(k-2)-2 \\
2
\end{array}
\right)+3t(3t(k-2)+2)$. If $s=3t$, then $G_2=G_3$. We get
$$e(G)\leq e(G_1)\leq e(G_2)=e(G_3)=\left(
\begin{array}{ccccccccc}
n-3t(k-2)-2 \\
2
\end{array}
\right)+3t(3t(k-2)+2),$$ a contradiction. Next, we consider $s\geq3t+1.$ Together with $n\geq\frac{3tk^3-(9t-2)k^2+(6t-5)k+6}{(k-2)^2}$, we have
\begin{align*}
&\left(
\begin{array}{ccccccccc}
n-3t(k-2)-2 \\
2
\end{array}
\right)+3t(3t(k-2)+2)-e(G_2)
\\=&\left(
\begin{array}{ccccccccc}
n-3t(k-2)-2 \\
2
\end{array}
\right)+3t(3t(k-2)+2)-\left(
\begin{array}{ccccccccc}
n-(k-2)s-2 \\
2
\end{array}
\right)-s(s(k-2)+2)
\\=&\frac{1}{2}(s-3t)(2(k-2)n-(k^2-2k)s-3tk^2+(6t-5)k+6)
\\ \geq&\frac{1}{2}(s-3t)(2(k-2)n-(k^2-2k)\frac{n-3}{k-1}-3tk^2+(6t-5)k+6)
\\=&\frac{s-3t}{2(k-1)}((k-2)^2n-(3tk^3-(9t-2)k^2+(6t-5)k+6))
\\ \geq&0.
\end{align*}
We get
$$e(G)\leq e(G_1)\leq e(G_2)\leq\left(
\begin{array}{ccccccccc}
n-3t(k-2)-2 \\
2
\end{array}
\right)+3t(3t(k-2)+2),$$ a contradiction.

This completes the proof.
$\hfill\square$\\

\section{Proof of Theorem \ref{th:1.2.} }
\hspace{1.3em}

Li and Feng \cite{LF} obtained a fundamental result to compare the adjacency spectral radius of a graph and its subgraph.

\noindent\begin{lemma}\label{le:4.1.}\cite{LF}
 Let $G$ be a connected graph and let $H$ be a subgraph of $G$. Then
\begin{align*}
\rho(G)\geq\rho(H),
\end{align*}
where the equality holds if and only if $G=H$.
\end{lemma}

Fan et al. gave the following lemma in \cite{FGHL}.

\noindent\begin{lemma}\label{le:4.2.}\cite{FGHL}
Let $n=\sum\limits_{i=1}^{t}n_{i}+s$. If $n_{1}\geq n_{2}\geq\cdots\geq n_{t}\geq p$ and $n_1<n-s-p(t-1)$. Then
\begin{align*}
\rho(K_{s}\vee(K_{n_{1}}\cup K_{n_{2}}\cup\cdots\cup K_{n_{t}}))<\rho(K_s\vee(K_{n-s-p(t-1)\cup(t-1)K_{p}})).
\end{align*}
\end{lemma}

Hong \cite{H} provided a upper bound on the adjacency spectral radius of a graph.

\noindent\begin{lemma}\label{le:4.3.}\cite{H}
Let $G$ be a graph with n vertices. Then
\begin{align*}
\rho(G)\geq\sqrt{2e(G)-n+1},
\end{align*}
where the equality holds if and only if $G$ is a star or a complete graph.
\end{lemma}

\noindent\textbf{Proof of Theorem \ref{th:1.2.}.}

Suppose that a connected $\frac{t}{t(k-2)+1}$-tough graph $G$ has no spanning $k$-tree, where $t\geq1$ is an integer and $k\geq1$. By Lemma \ref{le:2.1.}, we have
\begin{align*}
c(G-S)\geq(k-2)|S|+3
\end{align*}
for some nonempty subset $S$ of $V(G)$. According to the definition of $\frac{t}{t(k-2)+1}$-tough graphs, we obtain
\begin{align*}
\frac{|S|}{c(G-S)}\geq \tau(G)\geq\frac{t}{t(k-2)+1}.
\end{align*}
We conclude
$$(t(k-2)+1)|S|\geq t\cdot c(G-S)\geq t((k-2)|S|+3),$$
and hence $|S|\geq3t$. Let $|S|=s$. We know $G$ is a spanning subgraph of $G_1=K_{s}\vee(K_{n_{1}}\cup K_{n_{2}}\cup\cdots\cup K_{n_{(k-2)s+3}})$ for some positive integers $n_{1}\geq n_{2}\geq\cdots\geq n_{(k-2)s+3}$ and $\sum\limits_{i=1}^{(k-2)s+3}n_{i}=n-s$. By Lemma \ref{le:4.1.}, we have
$$\rho(G)\leq \rho(G_1).$$
Let $G_2=K_{s}\vee(K_{n-(k-1)s-2}\cup((k-2)s+2) K_{1})$, where $n\geq(k-1)s+3$. By Lemma \ref{le:4.2.}, we obtain
$$\rho(G_1)< \rho(G_2).$$
If $s=3t$, then $G_2=K_{3t}\vee(K_{n-3t(k-1)-2}\cup(3t(k-2)+2) K_{1})$. We get
$$\rho(G)\leq \rho(G_1)< \rho(G_2)=\rho(K_{3t}\vee(K_{n-3t(k-1)-2}\cup(3t(k-2)+2) K_{1})),$$
a contradiction. Next, we consider $s\geq3t+1$.
\begin{align}
2e(G_2)=&2\left(
\begin{array}{ccccccccc}
n-(k-2)s-2 \\
2
\end{array}
\right)+2s((k-2)s+2)\nonumber
\\=&2(k-2)s^2-(2kn-k^2-k-4n+2)s+n^2-5n+6.
\end{align}
It follows from $(3)$ and Lemma \ref{le:4.3.} that
\begin{align}
\rho(G_2)\leq&\sqrt{2e(G_2)-n+1}\nonumber
\\=&\sqrt{2(k-2)s^2-(2kn-k^2-k-4n+2)s+n^2-6n+7}.
\end{align}
Let $f(s)=2(k-2)s^2-(2kn-k^2-k-4n+2)s+n^2-6n+7$. Recall that $s\geq3t+1$ and $n\geq(k-1)s+3$. Hence, we obtain $3t+1\leq s\leq\frac{n-3}{k-1}$.
\begin{align*}
&f(3t+1)-f(\frac{n-3}{k-1})
\\=&\frac{(n-3kt-k+3t-2)((2k^2-8k+8)n-k^3-(6t+2)k^2+3(6t+2)k^2+3(6t+5)k-12t-8)}{(k-1)^2}
\\ \geq&\frac{(24t+4)k^3-(144t-8)k^2+4(66t-25)k-144t+120}{(k-1)^2}
\\ \geq&\frac{648t+108-1296t+72+792t-300-144t+120}{(k-1)^2}
\\ =&0.
\end{align*}
This implies that $f(s)\leq max\{f(3t+1), f(\frac{n-3}{k-1})\}\leq f(3t+1)$ for $3t+1\leq s\leq\frac{n-3}{k-1}$. Combining this with $(4)$, $t\geq1$, $k\geq3$ and $n\geq 3kt+k-3t+6$, we obtain
\begin{align*}
\rho(G_2)\leq&\sqrt{f(3t+1)}
\\=&\sqrt{(n-3t(k-2)-3)^2-(2k-4)n-(9t^2-3t-1)k^2+(54t^2-3t+3)k-72t^2+6t-8}
\\ \leq&\sqrt{(n-3t(k-2)-3)^2-(19t^2+3t+1)k^2+(4t^2+15t-5)k+72t^2+6t-16}
\\ \leq&\sqrt{(n-3t(k-2)-3)^2-(81t^2-12t+8)}
\\ <&n-3t(k-2)-3.
\end{align*}
Since $K_{n-3t(k-2)-2}$ is a proper subgraph of $K_{3t}\vee(K_{n-3t(k-1)-2}\cup(3t(k-2)+2) K_{1})$, we have
$$\rho(G)\leq\rho(G_1)\leq\rho(G_2)<n-3t(k-2)-2=\rho(K_{n-3t(k-2)-2})<\rho(K_{3t}\vee(K_{n-3t(k-1)-2}\cup(3t(k-2)+2) K_{1})),$$
a contradiction.

This completes the proof.
$\hfill\square$\\

\section{Proof of Theorem \ref{th:1.3.} }
\hspace{1.3em}

In the following, we introduce some lemmas concerning the signless Laplacian spectral radius.

\noindent\begin{lemma}\label{le:5.1.}\cite{SYZL}
 Let $G$ be a connected graph. If $H$ be a subgraph of $G$, then
\begin{align*}
q(G)\geq q(H),
\end{align*}
with equality holding if and only if $G=H$.
\end{lemma}

\noindent\begin{lemma}\label{le:5.2.}\cite{FLL}
Let $n=\sum\limits_{i=1}^{t}n_{i}+s$. If $n_{1}\geq n_{2}\geq\cdots\geq n_{t}\geq p$ and $n_1<n-s-p(t-1)$. Then
\begin{align*}
q(K_{s}\vee(K_{n_{1}}\cup K_{n_{2}}\cup\cdots\cup K_{n_{t}}))<q(K_s\vee(K_{n-s-p(t-1)}\cup(t-1)K_{p})).
\end{align*}
\end{lemma}

Das \cite{D} provided a upper bound on the signless Laplacian spectral radius of a graph.

\noindent\begin{lemma}\label{le:5.3.}\cite{D}
Let $G$ be a graph with n vertices. Then
\begin{align*}
q(G)\geq\frac{2e(G)}{n-1}+n-2.
\end{align*}
\end{lemma}

\noindent\textbf{Proof of Theorem \ref{th:1.3.}.}

Suppose that a connected $\frac{t}{t(k-2)+1}$-tough graph $G$ has no spanning $k$-tree, where $t\geq1$ is an integer and $k\geq1$. By Lemma \ref{le:2.1.}, we have
\begin{align*}
c(G-S)\geq(k-2)|S|+3
\end{align*}
for some nonempty subset $S$ of $V(G)$. According to the definition of $\frac{t}{t(k-2)+1}$-tough graphs, we obtain
\begin{align*}
\frac{|S|}{c(G-S)}\geq \tau(G)\geq\frac{t}{t(k-2)+1}.
\end{align*}
We conclude
$$(t(k-2)+1)|S|\geq t\cdot c(G-S)\geq t((k-2)|S|+3),$$
and hence $|S|\geq3t$. Let $|S|=s$. We know $G$ is a spanning subgraph of $G_1=K_{s}\vee(K_{n_{1}}\cup K_{n_{2}}\cup\cdots\cup K_{n_{(k-2)s+3}})$ for some positive integers $n_{1}\geq n_{2}\geq\cdots\geq n_{(k-2)s+3}$ and $\sum\limits_{i=1}^{(k-2)s+3}n_{i}=n-s$. By Lemma \ref{le:5.1.}, we have
$$q(G)\leq q(G_1).$$
Let $G_2=K_{s}\vee(K_{n-(k-1)s-2}\cup((k-2)s+2) K_{1})$, where $n\geq(k-1)s+3$. By Lemma \ref{le:5.2.}, we obtain
$$q(G_1)< q(G_2).$$
If $s=3t$, then $G_2=K_{3t}\vee(K_{n-3t(k-1)-2}\cup(3t(k-2)+2) K_{1})$. We get
$$q(G)\leq q(G_1)< q(G_2)=q(K_{3t}\vee(K_{n-3t(k-1)-2}\cup(3t(k-2)+2) K_{1})),$$
a contradiction. Next, we consider $s\geq3t+1$.
\begin{align}
2e(G_2)=&2\left(
\begin{array}{ccccccccc}
n-(k-2)s-2 \\
2
\end{array}
\right)+2s((k-2)s+2)\nonumber
\\=&2(k-2)s^2-(2kn-k^2-k-4n+2)s+n^2-5n+6.
\end{align}
It follows from $(5)$ and Lemma \ref{le:5.3.} that
\begin{align}
q(G_2)\leq&\frac{2e(G_2)}{n-1}+n-2\nonumber
\\=&\frac{2(k-2)s^2-(2kn-k^2-k-4n+2)s+2n^2-8n+8}{n-1}.
\end{align}
Let $g(s)=2(k-2)s^2-(2kn-k^2-k-4n+2)s+n^2-8n+8$. Recall that $s\geq3t+1$ and $n\geq(k-1)s+3$. Hence, we obtain $3t+1\leq s\leq\frac{n-3}{k-1}$.
\begin{align*}
&g(3t+1)-g(\frac{n-3}{k-1})
\\=&\frac{(n-3kt-k+3t-2)((2k^2-8k+8)n-k^3-(6t+2)k^2+3(6t+2)k^2+3(6t+5)k-12t-8)}{(k-1)^2}
\\ \geq&\frac{(3kt^2+k+18t+4-3kt-k+3t-2)((2k^2-8k+8)(3kt^2+k+18t+4)-k^3-(6t+2)k^2}{(k-1)^2}
\\ &+\frac{3(6t+5)k-12t-8)}{(k-1)^2}
\\ =&0.
\end{align*}
This implies that $g(s)\leq max\{g(3t+1), g(\frac{n-3}{k-1})\}\leq g(3t+1)$ for $3t+1\leq s\leq\frac{n-3}{k-1}$. Combining this with $(6)$, $t\geq1$, $k\geq3$ and $n\geq 3kt^2+k+18t+4$, we obtain
\begin{align*}
q(G_2)\leq&\frac{g(3t+1)}{n-1}
\\=&\frac{2n^2-2(3kt+k-6t+2)n+(3t+1)k^2+(18t^2+15t+3)k-36t^2-30t+2}{n-1}
\\=&2(n-3t(k-2)-3)-\frac{2(k-2)n-(3t+1)k^2-3(6t^2+3t+1)k+36t^2+18t+4}{n-1}
\\ \leq&2(n-3t(k-2)-3)-\frac{(6t^2-3t+1)k^2-(30t^2-27t-1)k+36t^2-54t+2}{n-1}
\\ \leq&2(n-3t(k-2)-3).
\end{align*}
Since $K_{n-3t(k-2)-2}$ is a proper subgraph of $K_{3t}\vee(K_{n-3t(k-1)-2}\cup(3t(k-2)+2) K_{1})$, we have
$$q(G)\leq q(G_1)\leq q(G_2)<n-3t(k-2)-2=q(K_{n-3t(k-2)-2})<q(K_{3t}\vee(K_{n-3t(k-1)-2}\cup(3t(k-2)+2) K_{1})),$$
a contradiction.

This completes the proof.
$\hfill\square$\\

\section{Concluding remark}
\hspace{1.3em}

It is challenging to determine what sufficient conditions on $\tau(G)$ guarantee the existence of a spanning $k$-tree in a $\tau$-tough graph.
Win \cite{W} established the result for $\tau(G)\geq\frac{1}{k-2}$. Liu, Fan and Shu \cite{LFS} obtained the results for $\tau=\frac{1}{k}$ and
$\tau=\frac{1}{k-1}$. In this paper, we present the result for $\frac{1}{k-2}>\tau\geq\frac{1}{k-1}$. We naturally raise the following problem.\\

\noindent\begin{problem}\label{pro:1.3.}
What is the spectral condition to guarantee the existence of a spanning $k$-tree among $\tau$-tough graphs, where $\frac{1}{k-1}>\tau>\frac{1}{k}$?\\
\end{problem}

We will continue to investigate this problem in future work.\\

\textbf{Declaration of competing interest}\\

The authors declare that they have no known competing financial interests or personal relationships that could have appeared to influence the work reported in this paper.\\

\textbf{Data availability}\\

No data was used for the research described in the article.

\end{document}